\documentclass[english,leqno]{article}
\usepackage{libertine}
\usepackage[T1]{fontenc}
\usepackage[utf8]{inputenc}
\usepackage{mathtools}
\usepackage{enumitem}
\usepackage{amsmath}
\usepackage{amsthm}
\usepackage{amssymb}
\usepackage{comment}
\usepackage[scr=rsfs]{mathalpha}
\usepackage{xargs}[2008/03/08]

\usepackage{microtype}
\usepackage{babel}
\usepackage[unicode=true,pdfusetitle, breaklinks=false,backref=false,colorlinks=false] {hyperref}

\theoremstyle{plain}
\newtheorem{theorem}{Theorem}
\newtheorem*{theorem*}{Theorem}
\newtheorem{fact}[theorem]{Fact}
\newtheorem*{fact*}{Fact}
\newtheorem{question}[theorem]{Question}
\newtheorem*{question*}{Question}
\newtheorem{lemma}[theorem]{Lemma}
\newtheorem*{lemma*}{Lemma}
\newtheorem{proposition}[theorem]{Proposition}
\newtheorem*{proposition*}{Proposition}
\newtheorem{corollary}[theorem]{Corollary}
\newtheorem*{corollary*}{Corollary}

\theoremstyle{remark}

\newtheorem*{remark*}{Remark}
\newtheorem{claim}[theorem]{Claim}
\newtheorem*{claim*}{Claim}

\theoremstyle{definition}
\newtheorem{definition}[theorem]{Definition}

\makeatletter
\@ifundefined{date}{}{\date{}}

\newlength{\@thlabel@width}%
\newcommand{\thmenumhspace}{\settowidth{\@thlabel@width}{\itshape1.}\sbox{\@labels}{\unhbox\@labels\hspace{\dimexpr-\leftmargin+\labelsep+\@thlabel@width-\itemindent}}}

\makeatother

\newcommand\fii{\varphi}%
\newcommand\zfc{\mathrm{ZFC}}%
\newcommand\zf{\mathrm{ZF}}%
\newcommand\LL{\mathcal{L}}%
\newcommand\K{\mathcal{K}}%
\newcommand\A{\mathcal{A}}%

\newcommand\hod{\operatorname{HOD}}%
\newcommand\Def{\operatorname{Def}}%

\newcommand\dom{\operatorname{Dom}}%

\newcommand\ord{{\bf Ord}}%

\newcommandx\sing[1][usedefault, addprefix=\global, 1=]{{\bf Sing}_{#1}}%

\newcommand\con{\subseteq}%
\newcommand\emp{\varnothing}%
\newcommand\aaa{\mathtt{aa}}%

\DeclareUnicodeCharacter{0394}{$\Delta$}
\newcommand\dq{\Delta{(Q_{1})}}
\newcommand\cq{C(\Delta(Q_{1}))}
\newcommand\lqw{\LL{(Q_{1},\W )}}
\newcommand\dqw{\Delta{(Q_{1},\W )}}
\newcommand\sqw{\Sigma{(Q_{1},\W )}}
\newcommand\pqw{\Pi{(Q_{1},\W )}}
\newcommand\cqw{C(\Delta(Q_{1},\W ))}
\DeclareMathOperator{\Sat}{Sat}
\DeclareMathOperator{\Fml}{Fml}
\DeclareMathOperator{\Fr}{Fr}
\DeclareMathOperator{\Ind}{Ind}
\DeclareMathOperator{\Asn}{Asn}
\newcommand{\W}{\mathrm{W}}
\newcommand\LST{\textup{LST}}
\newcommand\lst{\mathscr{L}}
\begin{document}
\title{Inner models from extended logics and the Delta-operation%
	\thanks{We are  grateful to Menachem Magidor for suggestions concerning the results of this paper.
		We would like to thank the anonymous referee for their comments and suggestions.
		We thank the American Institute of Mathematics for their support.
		This project has received  funding from the European Research Council (ERC) under the European Union’s Horizon 2020 research and innovation programme (grant agreement No 101020762) as well as from the Research Council of Finland (grants number 322795 and 368671).}
}
\author{Jouko  Väänänen\\ {\small Helsinki and Amsterdam} \and  Ur Ya'ar\\ {\small Helsinki }}
\maketitle
\begin{abstract}
	If $\LL$ is an abstract logic (a.k.a. model theoretic logic), we can define the inner model $C(\LL)$ by replacing first order logic with $\LL$ in G\"odel's definition of the inner model $L$ of constructible sets.
	Set theoretic properties of such inner models $C(\LL)$ have been investigated recently and a spectrum of new inner models is emerging between $L$ and $\hod$. The topic of this paper is the effect on $C(\LL)$ of a slight modification of $\LL$ i.e. how sensitive is $C(\LL)$ on the exact definition of $\LL$? The $\Delta$-extension $\Delta(\LL)$ of a logic is generally considered a ``mild'' extension of $\LL$.
	We give examples of logics $\LL$ for which the inner model $C(\LL)$ is consistently strictly smaller than the inner model $C(\Delta(\LL))$, and in one case we show this follows from the existence of $0^{\sharp}$.
\end{abstract}

\section{Introduction}

This paper is part of a larger project focusing on applications of strong logics in set theory. A pioneering result in this direction is Menachem Magidor's characterization of L\"owenheim-Skolem type properties as well as compactness properties of second order logic in terms of supercompact and extendible cardinals \cite{MR295904}. Another example is the so-called Magidor-Malitz quantifier (\cite{MR453484}). A more recent example is the project of ``extended constructibility" (\cite{KMV1,KMV2}).

In \cite{KMV1} an inner model of set theory $C(\LL)$ was associated with any abstract logic $\LL$
(all relevant notions are defined below).
For many logics $\LL$ this inner model is just G\"odel's inner model $L$ of constructible sets.
But consistently, and also assuming large cardinals, $C(\LL)$ can be very different from Gödel's $L$ \cite{KMV1}.
For example, if $\LL$ is the stationary logic of \cite{StatLogic} and we assume that there is a proper class of Woodin cardinals, the resulting inner model $C(\LL)$, denoted $C(\mathtt{aa})$, has a proper class of measurable cardinals \cite{KMV2}.
Associated with any abstract logic $\LL$ there is an extension $\Delta(\LL)$ of  $\LL$, introduced in \cite{Barwise-Axioms}, see also \cite{MSS-Delta}. We give the definition below.
It is generally considered that the difference between $\Delta(\LL)$ and  $\LL$ is minor, from the model theoretic point of view, in the sense that $\Delta$ preserves a lot of model theoretic properties of the logic $\LL$.
In this paper we compare the inner models $C(\LL)$ and $C(\Delta(\LL))$ and show that these can be different
(of course, if $V=L$ they are the same, so we look for consistency results and results under large cardinal assumptions).
This shows that the set-theoretical power of  $\Delta(\LL)$ and  $\LL$ can be quite different even if they greatly resemble each other from a model theoretic view point.
We begin by recalling definitions and facts we will use throughout the paper.

\subsection{Abstract Logics and the $\Delta$-operation}\label{subsec:prelimabst}

We present the basic notions we require from the theory of abstract logics. For further information, see \cite{MT-Logics}.

An \emph{abstract logic} $\LL$ consists of a pair $(S,T)$ where, given a vocabulary $\tau$,
$S$ defines the class of \emph{sentences} of $\LL$ in the vocabulary $\tau$,
and $T$ defines the \emph{satisfaction relation} between $\tau$-structures and members of $S$.
We usually say simply ``a logic'' instead of ``abstract logic'', abuse notation by writing e.g. $\varphi \in \LL$ instead of $S(\varphi)$,
and write $\mathcal{M} \models_{\LL} \varphi$ instead of $T(\mathcal{M},\varphi)$.
We may consider \emph{formulas} with free variables by adding constant symbols to the vocabulary.
While this definition is quite general, prominent examples of abstract logics are concretely given by supplementing the syntax of first-order logic with additional quantifiers and extending the semantics with appropriate clauses for these quantifiers (some examples are presented below).
A core idea of the field of abstract logics, giving it the synonym ``Model Theoretic'' logics, is looking at classes of models as defining logics.
We recall some definitions.

\begin{definition}\label{def:Delta}
	A \emph{model class} is any class of models of a fixed vocabulary $\tau$ which is  closed under isomorphisms.
	Every sentence $\phi$ of a logic $\LL$ defines a model class, namely the class of models satisfying $\phi$.
	Suppose $\mathcal{K}$ is a model class of vocabulary $\tau$.
	$\mathcal{K}$ is said to be:
	\begin{itemize}
		\item \emph{$\LL$-definable} if there is a sentence $\phi$ of $\LL$ such that  $\mathcal{K}$ is the class of models satisfying $\phi$;
		\item     \emph{$\Sigma(\LL)$-definable} if there is a bigger vocabulary $\tau'$, possibly many-sorted, and a sentence $\phi$ of $\LL$ in the vocabulary $\tau'$ such that $\mathcal{K}$ is the class of reducts of models of $\phi$ to $\tau$;
		\item \emph{$\Pi(\LL)$-definable} if its complement  (in the class of all models of vocabulary $\tau$) is $\Sigma(\LL)$-definable;
		\item \emph{$\Delta(\LL)$-definable} if both $\mathcal{K}$ and its complement  are $\Sigma(\LL)$-definable (in other words, if it is both $\Sigma(\LL)$- and $\Pi(\LL)$-definable).
	\end{itemize}

	A logic $\LL'$ \emph{extends} another logic $\LL$, in symbols $\LL\le \LL'$, if every definable model class of $\LL$ is also a definable model class of $\LL'$.
	The  logics $\LL$ and $\LL'$ are \emph{equivalent}, in symbols $\LL \equiv \LL'$, if $\LL\le\LL'$ and $\LL'\le\LL$.

	The \emph{$\Delta$-extension} $\Delta(\LL)$ of an abstract logic $\LL$ is the $\leq$-least logic  $\LL^*$ such that every $\Delta(\LL)$-definable model class is $\LL^*$ definable.
\end{definition}

In the most abstract setting, $\Delta(\LL)$ can be defined as the logic whose sentences are simply the $\Delta(\LL)$-definable model classes, so the satisfaction relation is just the membership relation.
More concretely, we can follow \cite[Definition II.7.2.3]{MT-Logics} and let the sentences of $\Delta(\LL)$ in a vocabulary $\tau$ be pairs $(\phi,\psi)$ such that $\phi$ and $\psi$ are $\LL$-sentences in vocabularies extending  $\tau$, and:
\begin{equation}\label{delta}
	\begin{array}{l}
		\mbox{For all models $\mathcal{M}$ of vocabulary $\tau$:}              \\
		\mbox{$\mathcal{M}$ is a reduct of a model of $\phi$  if and only if } \\
		\mbox{$\mathcal{M}$ is not a reduct of a model of $\psi$.}
	\end{array}\end{equation}

Intuitively speaking,
$\Delta(\LL)$-definability is a form of \emph{implicit} definability, and is related to the notion of \emph{interpolation}:

\begin{definition}\label{def:interpolation}
	A logic $\LL$ is said to have the \emph{Craig Interpolation Property} if every two disjoint $\Sigma(\LL)$-definable model classes (of the same vocabulary) can be separated by a definable model class of $\LL$.
	$\LL$ has the  \emph{Suslin-Kleene Interpolation Property} if every two \emph{complementary} $\Sigma(\LL)$-definable classes can be separated by a definable class of $\LL$.
\end{definition}
In other words, $\LL$ has the Suslin-Kleene Interpolation Property iff every  $\Delta(\LL)$-definable class is actually definable iff $\Delta(\LL)=\LL$.
Clearly every logic with the Craig Interpolation Property has the Suslin-Kleene Interpolation Property, but the converse is generally false (\cite[Theorem 3.1]{Hutchinson-ModelTheorySet1976}).

$\Delta(\LL)$ can be considered a minor extension of $\LL$ in the sense that it preserves many model theoretic properties of $\LL$:
\begin{fact}[\cite{MSS-Delta}]\label{fact:D-ext-properties}
	Let $\LL$ be an abstract logic. Then
	\begin{itemize}
		\item $\Delta(\LL)$ is $(\kappa,\lambda)$-compact (for infinite cardinals $\kappa\geq \lambda$) if $\LL$ is;
		      \begin{definition}\label{def:compact}
			      A logic $\LL$ is said to be \emph{$(\kappa, \lambda)$-compact} whenever given a set of $\LL$-sentences $\Sigma$ of cardinality $\kappa$, if each subset of $\Sigma$ of cardinality $<\lambda$ has a model, then $\Sigma$ has a model.
			      A logic is \emph{$\kappa$-compact} if it is $(\kappa,\aleph_{0})$-compact.
		      \end{definition}
		\item $\Delta(\LL)$ has the same L\"owenheim-number as $\LL$;
		      \begin{definition}\label{def:lowenheim}
			      The \emph{Löwenheim-number} of $\LL$ is the smallest cardinal $\kappa$ such that if an arbitrary sentence of $\LL$ has any model, the sentence has a model of cardinality no larger than $\kappa$.
		      \end{definition}
		\item $\Delta(\LL)$ is axiomatizable if $\LL$ is;
		\item $\Delta(\LL)$ is the minimal extension of $\LL$ to a logic with the Suslin-Kleene Interpolation Property;
		\item $\Delta(\Delta(\LL))=\Delta(\LL)$.
	\end{itemize}
\end{fact}

A weakness of the $\Delta$-operation is that its syntax depends, a priori, on set theory.
Depending on the logic $\LL$, the question whether a pair $(\phi,\psi)$ is a sentence of $\Delta(\LL)$, i.e. whether the equivalence (\ref{delta}) above holds, may be non-absolute in set theory and of high set theoretical complexity.
This is in sharp contrast to the absoluteness of the syntax of most logics that one encounters, e.g. first order logic and its extensions by generalized quantifiers.
In fact, we exploit a non-absolute case of (\ref{delta}) in our proof of Theorem \ref{thm:DeltaQ1} below.
\label{AC} One consequence of the potential non-absoluteness of the syntax of $\Delta(\LL)$ is that the usual method of proving that the Axiom of Choice (AC) holds in models of the form $C(\LL)$
(Fact \ref{fact:AC})
cannot be used, as it requires having a recursive set of formulas.
However our results do not require AC to hold in the models we discuss, so we leave the question whether or not it holds open.

\subsection{Inner models from extended logics}\label{subsec:IMEL}

We review some of the definitions and facts from \cite{KMV1} regarding constructing inner models using extended logics.

\begin{definition}\label{def:C(L)}
	Let $\LL$ be a logic (extending first-order logic).
	\begin{enumerate}
		\item Let $M$ be a set. Define:
		      \[
			      \Def_{\LL}\left(M\right)=\left\{ \left\{ a\in M\mid\left(M,\in\right)\models_{\LL}\varphi(a,\bar{b})\right\} \mid\varphi\in\LL;\,\bar{b}\in M^{<\omega}\right\}
		      \]
		\item $C\left(\LL\right)$, the class of \emph{$\LL$-constructible sets},
		      is defined by induction:
		      \begin{align*}
			      L'_{0}            & =\emp                                                                       \\
			      L'_{\alpha+1}     & =\Def_{\LL}\left(L'_{\alpha}\right)                                         \\
			      L_{\beta}'        & =\bigcup_{\alpha<\beta}L_{\alpha}'\,\,\,\text{for limit \ensuremath{\beta}} \\
			      C\left(\LL\right) & =\bigcup_{\alpha\in Ord}L'_{\alpha}
		      \end{align*}
	\end{enumerate}
\end{definition}

Note that if $\LL$ is first-order logic, then we get exactly G\"{o}del's
$L$.
It is also easy to see that if $\LL\le\LL'$, then $C(\LL)\subseteq C(\LL')$, but, of course, not conversely.

\begin{fact}[\cite{KMV1} Proposition 2.3]
	For any logic $\LL$, $C\left(\LL\right)$ is a transitive model of
	$\mathrm{ZF}$ containing all the ordinals.
\end{fact}

\begin{fact}[\cite{KMV1} Proposition 2.9] \label{fact:C(L)bounds}
	If the logic $\LL$ is defined by hereditarily ordinal definable parameters, and each of its formulas is hereditarily ordinal definable (with finitely many free variables), then  $C(\LL) \subseteq \hod$.
\end{fact}

\begin{definition}\label{def:absoluteL}
	Suppose $A$ is any class and $T$ is any theory in the language of set theory.
	A logic $\mathcal{L}^*$ is \emph{$T$-absolute} if there are a $\Sigma_1$-predicate $S_1(x)$, a $\Sigma_1$-predicate $T_1(x, y)$, and a $\Pi_1$-predicate $T_1^{\prime}(x, y)$ such that $\varphi \in \mathcal{L}^* \Leftrightarrow S_1(\varphi)$, $M \models \varphi \Leftrightarrow T_1(M, \varphi)$ and $T \vdash \forall x \forall y\left(S_1(x) \rightarrow\left(T_1(x, y) \leftrightarrow T_1^{\prime}(x, y)\right)\right)$.
	If parameters from a class $A$ are allowed, we say that $\mathcal{L}^*$ is absolute \emph{with parameters from $A$}.

	Furthermore, $\mathcal{L}^*$ has \emph{$T$-absolute syntax} if its sentences are (coded as) natural numbers and there is a $\Pi_1$-predicate $S_1^{\prime}(x)$ such that $T \vdash \forall x\left(S_1(x) \leftrightarrow S_1^{\prime}(x)\right)$.
	We may allow parameters as above.
\end{definition}

\begin{fact}[\cite{KMV1} Theorem 3.4]\label{fact:absoluteL}
	Suppose $\mathcal{L}^*$ is $\zfc$-absolute with parameters from $L$, and the syntax of $\mathcal{L}^*$ is $\zfc$-absolute with parameters from $L$. Then $C\left(\mathcal{L}^*\right)=L$.
\end{fact}

Recall that an \emph{admissible set} is a transitive set $X$ such that  $\left<X,\in \right>$ satisfies $\mathrm{KP}$ -- the Kripke-Platek axioms of set theory.

\begin{definition}\label{def:adequate}
	A logic $\mathcal{L}^*$ is \emph{adequate to truth in itself}  if for all finite vocabularies $\tau$ there is function $\varphi \mapsto\ulcorner\varphi\urcorner$ from all formulas $\varphi(x_1, \ldots, x_n) \in \mathcal{L}^*$ in the vocabulary $\tau$ into $\omega$, and a formula $\operatorname{Sat}_{\mathcal{L}^*}(x, y, z)$ in $\mathcal{L}^*$ such that:
	\begin{enumerate}
		\item The function $\varphi \mapsto\ulcorner\varphi\urcorner$ is one-to-one and has a recursive range.

		\item For all admissible sets $M$, formulas $\varphi$ of $\mathcal{L}^*$ in the vocabulary $\tau$, structures $\mathcal{N} \in M$ in the vocabulary $\tau$, and $a_1, \ldots, a_n \in N$, the following conditions are equivalent:
		      \begin{enumerate}
			      \item $\langle M,\in \rangle \models \operatorname{Sat}_{\mathcal{L}^*}\left(\mathcal{N},\ulcorner\varphi\urcorner,\left\langle a_1, \ldots, a_n\right\rangle\right)$
			      \item $\mathcal{N} \models \varphi\left(a_1, \ldots, a_n\right)$.
		      \end{enumerate}

	\end{enumerate}
	We may admit ordinal parameters in this definition.
\end{definition}

\begin{fact}[\cite{KMV1} Proposition 2.7]\label{fact:AC}
	If $\mathcal{L}^*$ is adequate to truth in itself, then $C\left(\mathcal{L}^*\right)$ satisfies the Axiom of Choice.
\end{fact}

\section{Initial results}\label{sec:initial}

The standard example of the use of the $\Delta$-operation is the result of \cite{Barwise-Axioms} to the effect that $\Delta(\LL(Q_0))=\LL_{\mbox{\tiny HYP}}$,
where $Q_{0}$ is the quantifier asserting ``there are $\geq \aleph_{0}$ many...'' and
$\LL_{\mbox{\tiny HYP}}$ is the smallest admissible fragment of $\LL_{\omega_1\omega}$.
This result implies that $\LL(Q_0)\ne\LL_{\mbox{\tiny HYP}}$ (as Barwise \cite{Barwise-Axioms} shows in the proof of Corollary 4.3).
By fact \ref{fact:absoluteL} $C(\LL(Q_0))=C(\LL_{\mbox{\tiny HYP}})=L$, as both $\LL(Q_0)$ and $\LL_{\mbox{\tiny HYP}}$ are absolute logics (by \cite{Barwise-Absolute}).
In conclusion, $\Delta(\LL(Q_0))\ne\LL(Q_0)$ but $C(\Delta(\LL(Q_0)))=C(\LL(Q_0))$.

Another  example of this phenomenon is second order logic $\LL^2$. Assuming the Axiom of Choice, the inner model $C(\LL^2)$ is equal to $\hod$ (\cite{MyhillScott}),  and therefore also  $C(\Delta(\LL^2))$ is equal to $\hod$ by Fact \ref{fact:C(L)bounds}
(note that the definition of $\Delta(\LL^{2})$ doesn't require any parameters which aren't hereditarily ordinal definable).
Hence $C(\Delta(\LL^2))=C(\LL^2)$, but $\Delta(\LL^2)\ne\LL^2$ by \cite{Craig-Satisfaction1965} (which essentially shows that the satisfaction relation for $\LL^{2}$ is a counterexample).
Note that $C(\LL^2)$ may be a proper subclass of  $\hod$ in the absence of AC (\cite{MR472523}).

Thus we have established:
\begin{proposition}
	There are logics $\LL$ such that $\Delta(\LL)\ne\LL$ but provably, $C(\Delta(\LL))=C(\LL)$.
\end{proposition}

We now show that there are logics where the $\Delta$-operation might not preserve the inner model $C(\LL)$.
Let $H$ be the Henkin quantifier \cite{Henkin-infiniteform}
\begin{equation}\label{eq:henqin-qua}
	Hxyuv\ \phi(x,y,u,v,\vec{w}) =	\left(
	\begin{array}{cc}
			\forall x & \exists y \\
			\forall u & \exists v
		\end{array}\right)\phi(x,y,u,v,\vec{w})
\end{equation}
the meaning of which is
\begin{equation}\label{eq:henkin-so}
	\exists f\exists g
	\forall x\forall u
	\,\phi(x,f(x),u,g(u),\vec{w}).
\end{equation}
Let us use $C(H)$ to denote the inner model $C(\LL(H))$. Recall that  $\hod_1$ is defined as $C(\Sigma^1_1)$ (\cite{KMV1}), where $\Sigma^1_1$ is  existential second order logic, i.e. the logic of formulas of second order logic which start with an existential second order quantifier and then continue with a first order formula.
\begin{theorem} \label{thm:HenkinHOD1}
	$C(H)=\hod_1$.
\end{theorem}

\begin{proof}
	By \cite[Corollary 4.6]{Walkoe},
	in the presence of a pairing function
	any $\Sigma^1_1$-formula can be expressed in terms of the Henkin quantifier $H$.
	This implies $\hod_1\subseteq C(H)$.
	Conversely, any formula starting with $H$ can be translated via (\ref{eq:henkin-so}) into a $\Sigma^1_1$-form.
	This makes it possible to prove by induction on
	the structure of the  formula defining a set $X \in C(H)$
	that $C(H)\subseteq \hod_1$.
	For example if $X$ is defined using $\neg Hxyuv\ \phi$,
	then its complement relative to some level of the construction is, by induction, in $\hod_{1}$,
	and so $X$ will enter $\hod_{1}$ in the next level.
\end{proof}

Since our $\Delta$-operation is  many-sorted, $\Delta(\LL(H))$ contains full second order logic \cite{AKL-Henkin}. The logic $\LL(H)$ is not $\Delta$-closed, i.e. $\Delta(\LL(H))\ne \LL(H)$ (\cite{GostanianHrbacek}), but we do not need this here.
By \cite[Theorem 7.9]{KMV1}  it is  consistent, relative to the consistency of $\mathrm{ZF}$, that  $\hod_1\ne \hod$. Hence:

\begin{theorem}\label{thm:Henkin-DHenkin}
	$Con(\mathrm{ZF})$ implies $Con(C(\LL(H))\ne C(\Delta(\LL(H))))$.
\end{theorem}

\section{The quantifier $Q_{1}$} \label{sec:Q1}

As a second example, let $Q_1$ be the quantifier asserting ``there are $\geq \aleph_1$ many...''. It is well-known that $\Delta(\LL(Q_1))\ne \LL(Q_1)$, as demonstrated by the property of an equivalence relation having uncountably many uncountable classes (see \cite[Chapter II section 7.2]{MT-Logics}).
This raises the question, what is the relationship between $C(\LL(Q_{1}))$ and
$C(\Delta(\LL(Q_{1})))$?

\begin{theorem}\label{thm:DeltaQ1}
	$Con(\mathrm{ZF})$ implies $Con(C(\LL(Q_{1}))\ne C(\Delta(\LL(Q_{1}))))$.
\end{theorem}

The proof will use the following machinery.
Recall that an \emph{Aronszajn tree} is a tree of height $\omega_1$ such that all its levels are countable, and it has no uncountable chain (a \emph{branch});
a \emph{Suslin tree} is an Aronszajn tree which has no uncountable antichain;
and an Aronszajn tree $S$ is called \emph{special} if there is a \emph{specializing function} -- a function $c:S \to A$ such that $A$ is countable, and for every  $q,p \in S$,  if $c(q)=c(p)$ then  $q,p$ are incompatible in $S$.
Note that a special Aronszajn tree must have an uncountable chain, so cannot be Suslin.
In general, there might be Aronszajn trees that are neither special nor Suslin.

A Suslin tree can be considered as a forcing notion which has the countable chain condition (ccc), and adds a generic uncountable chain to itself.
Furthermore, if $S,T$ are Suslin trees such that the tree product $S \otimes T$ is Suslin as well, then forcing with one preserves the Suslinity of the other.
Note that, since the tree product $S \otimes T$ is dense in the forcing product $S \times T$, this corresponds to the following well-known fact (see \cite[Lemma 16.6]{Jech}):
\begin{fact}\label{fact:ccc-prod}
	If $P$ and $Q$ are ccc forcing notions, then $P \times Q$ satisfies ccc  iff $P \Vdash$ ``$\check{Q}$ satisfies ccc'' iff $Q \Vdash$ ``$\check{P}$ satisfies ccc''.
\end{fact}

On the other hand,
given an $\omega_{1}$ tree $T$, we denote by $\mathcal{S}(T)$ the standard forcing notion consisting of finite partial specializing functions from $T$ to $\omega$. 
This forcing notion, defined in \cite{Baumgartner.Malitz.ea-EmbeddingTreesRationals1970}, specializes $T$,
 and if $T$ is Aronszajn, then  $\mathcal{S}(T)$ satisfies ccc.
Hence, if $T,S$ are Suslin trees such that  $T \otimes S$ is Suslin (i.e. ccc), then in particular $T \Vdash$ ``$\check{S}$ is Aronszajn'',
hence  $T \Vdash $ ``$\mathcal{S}(\check{S})$ is ccc'' (note that the definition of $\mathcal{S}(S)$ is absolute between models of set theory).
 By Fact \ref{fact:ccc-prod} this implies
that $\mathcal{S}(S) \Vdash$ ``$\check{T}$ satisfies ccc'' (i.e. it is Suslin).

To summarize, we have the following properties:
\begin{lemma}\label{lem:independent-trees}
	If $T,S$ are Suslin trees such that  $T \otimes S$ is Suslin, then
	\begin{enumerate}
		\item $T \Vdash$ ``$\check{S}$ is Suslin, $\check{T}$ is not Aronszajn'';
		\item $\mathcal{S}(T) \Vdash$ ``$\check{S}$ is Suslin, $\check{T}$ is special''.
	\end{enumerate}
\end{lemma}

Suslin trees satisfying the conclusion of the lemma are called \emph{independent}.
%
%
%
%
By \cite[Theorem 7.3]{abraham-shelah}, there is a sequence of Suslin trees $\mathcal{T}=\left\langle T_{n}\mid n<\omega\right\rangle \in L$ such that for every distinct $i_1,...,i_k \in \omega$, the product tree $T_{i_1}\otimes ... \otimes T_{i_k}$ is Suslin.
To summarize, for every $i \in \omega$:
\begin{itemize}
	\item $T_{i} \Vdash $ ``$\check{T}_{i}$ is not Aronszajn'';
	\item $\mathcal{S}( T_{i} ) \Vdash$ ``$\check{T}_{i}$ is special'';
    \item $T_{i}, \mathcal{S}(T_{i}) \Vdash $ ``for every distinct $i_1,...,i_k \in \omega \smallsetminus \{i\}$, the product tree $T_{i_1}\otimes ... \otimes T_{i_k}$ is Suslin''.
\end{itemize}
So, all in all, we can either force a branch or specialize $T_{i}$
without affecting the Suslinity of any of the other $T_{j}$s or the independence between them.
Note that all posets used, as well as their finite products, are ccc, and as a consequence a finite support product of them preserves $\omega_{1}$.

\begin{proof}[Proof of theorem \ref{thm:DeltaQ1}]
	By \cite[section 3]{KMV1}, $\zf \vdash C(\LL(Q_{1}))=L$.
	So we wish to construct a model in which $C(\Delta(\LL(Q_{1})))\ne L$.

	Let $a \subset \omega$ be a Cohen real over $L$.
	Using a finite support product, we force over $L[a]$ to add a branch to any $T_{i}$ for $i \in a$ (using $T_{i}$ itself) and specialize any $T_{j}$ for $j\notin a$ (using $\mathcal{S}(T_{j})$).
	As we noted, this product preserves $\omega_{1}$.
	Then we force Martin's Axiom ($\mathsf{MA}$) in the standard way (as in \cite{Solovay.Tennenbaum-IteratedCohenExtensions1971}) over the resulting model,
	and	denote the final model $W$.
	Note that still, $\omega_{1}^W = \omega_{1}^{L}$.
	We now work in $W$.
	As a consequence of $\mathsf{MA}$, any Aronszajn tree in $W$ is special (\cite{Baumgartner.Malitz.ea-EmbeddingTreesRationals1970}), and we have 
	\[
		a=\left\{ i<\omega\mid T_{i}\text{ is not Aronszajn}\right\}.
	\]

	\begin{claim}
		$a\in C(\Delta(\LL(Q_{1})))$.
	\end{claim}
	\begin{proof}
		Assume that, in the notation of Definition \ref{def:C(L)},
		$C(\Delta(\LL(Q_{1})))=\bigcup_\alpha L_{\alpha}'$, and let $\alpha>\omega_1$ be such that $\mathcal{T}\in L_{\alpha}'$.
		We can assume that  $T_{i}\con\omega_{1}$ for all $i$.
		Then $a$ is the set of all $i<\omega$ such that $L_{\alpha}'\models$``$T_{i}$ is not an Aronszajn tree'', so we need to show that ``$T\con\omega_{1}$ is a non-Aronszajn tree'' is expressible in  $\Delta(\LL(Q_{1}))$ over $L_{\alpha}'$.
		First note that if $T\in L_{\alpha}'$ and $T\con\omega_{1}$ then ``$T$ is an $\aleph_{1}$ tree'' is first order definable, and since $\omega_{1}^{L}=\omega_{1}^{W}$, $L_{\alpha}'$ is correct about this.
		Now by definition, an $\aleph_{1}$ tree is not Aronszajn iff it has a cofinal branch iff it has an uncountable branch.
		Having an uncountable branch is $\Sigma(\LL(Q_{1}))$:
		\[
			\exists B\left(\forall x,y\in T(Bx\land By\to(x\leq y\lor y\leq x))\land Q_{1}xBx\right).
		\]
		Since by $\mathsf{MA}$ all Aronszajn trees are special, an $\aleph_{1}$-tree is not Aronszajn also iff it is not special. But being a special tree is $\Sigma(\LL(Q_{1}))$:
		\footnote{Here we use $\omega$ as a parameter from $L'_{\alpha}$,
		but this can be avoided as we have $Q_1$ at our disposal, and all we need is that the range of $F$ is countable.}
		\[
			\exists F:T\to\omega\left(\forall x,y\in T(f(x)=f(y)\to x\bot y)\right).
		\]
		So under $\mathsf{MA}$, the property of being an Aronszajn tree is $\Delta(\LL(Q_{1}))$. Hence indeed, in $W$,  $a \in C(\Delta(\LL(Q_{1})))$.
	\end{proof}
	To conclude, as $a$ is Cohen generic over $L$, we obtain
	\[W \models C(\LL(Q_{1}))\ne C(\Delta(\LL(Q_{1}))). \qedhere
	\]
\end{proof}

\section{Symbiosis} \label{sec:symbiosis}

As a next step it would be natural to ask whether certain properties of a logic $\LL$ might imply that  $C(\LL)=C(\Delta(\LL))$ even though $\LL \ne \Delta(\LL)$.
One such candidate is the notion of symbiosis \cite{Vaananen-AbstractI}.
The idea of symbiosis is to establish a close connection between an abstract logic and predicates of set theory.
The motivation for this is the fact that many abstract logics seem to depend to a great extent on set theory, so it makes sense to ask exactly how strong is this dependence.
A paramount example is second order logic and its symbiosis with the power-set operation of set theory.
\begin{definition}\label{def:symbiosis}
	A set-theoretic relation $R$ and a logic $\LL^*$ are \emph{symbiotic} if the following conditions are satisfied:
	\begin{enumerate}
		\item Every $\LL^*$-definable model class is $\Delta_1(R)$-definable.
		\item Every $\Delta_1(R)$-definable model class is $\Delta(\LL^*)$-defin\-able.
	\end{enumerate}
\end{definition}

\noindent As noted in \cite{BagariaVaananen-Symbiosis}, the second clause can be replaced with
\begin{enumerate}
	\item[$\text 2^{*}$.]  The model class
	      \[
		      Q_{R}:=\left\{ \A\mid\A\cong\left(M,\in,\vec{a}\right),M\text{ transitive and }R(\vec{a})\right\}
	      \]
	      is $\Delta(\LL^*)$-definable.
\end{enumerate}

\begin{proposition}\label{prop:notsymbiotic}
	$\LL(Q_{1})$ is not symbiotic with any set theoretic predicate.
\end{proposition}
\begin{proof}
	Symbiosis implies that the well-ordering quantifier $\W $, defined by
	\[
		\mathcal{M} \models \W xy \varphi(x,y) \iff \{ (a,b)\in M^{2} \mid \mathcal{M} \models \varphi(a,b) \} \text{ is a well-order},
	\]
	is $\Delta(\LL)$-definable (cf. \cite[pg. 400]{Vaananen-AbstractI}).
	But the $\Delta$-operation preserves compactness (Fact \ref{fact:D-ext-properties}),
	and $\LL(Q_{1})$ is $\aleph_{0}$-compact (see e.g. \cite[II.3]{MT-Logics}),
	so $\Delta(\LL(Q_{1}))$ is $\aleph_{0}$-compact as well,
	which in particular implies that the well-ordering quantifier is not definable.
	So $\LL(Q_{0})$ cannot be symbiotic with any set-theoretic relation.
\end{proof}

However this can be easily amended by adding $\W $ to the logic, since $C(\LL(Q_{1},\W ))=L$ as well (again by \cite[section 3]{KMV1}), and we do have:

\begin{proposition}
	The logic $\LL(Q_{1},\W )$ is symbiotic with the predicate
	$R_{\aleph_{1}}$ (where $R_{\aleph_{1}}(a)$ $\iff$ $\left|a\right|\ge\aleph_{1}$).
\end{proposition}

\begin{proof}
	Denote $\LL=\LL(Q_{1},\W )$ and $R=R_{\aleph_{1}}$.
	We verify clauses $1$ and  $2^*$ above.

	\noindent
	1) Every $\LL$-definable model class is $\Delta_{1}(R)$-definable:

	\noindent Let $\K=\left\{ \A\mid\A\models_{\LL} \fii\right\} $ for some $\fii\in\LL$.
	Denote by $\zfc_{n}^{-}$ the axioms of $\zfc$ without the Power Set Axiom, and with the axiom schemata of Separation and Collection restricted to  $\Sigma_{n}$ formulas.
	$\zfc_{n}^{-}$ is finitely axiomatizable.
	Then there is some $n<\omega$ such that the following are equivalent:
	\begin{enumerate}[label=\alph*)]
		\item $\A\in\K$
		\item $\exists M$ $\big($ $M$ is transitive $\land$ $M\models\zfc_{n}^{-}$
		      $\land$ $\forall a\in M$ ( $R(a)$ $\leftrightarrow$  $M\models\left|a\right|\geq\aleph_{1}$
		      ) $\land$ $\A\in M$ $\land$ $M\models$
		      ``$\A\models_{\LL} \fii$'' $\big)$.
		\item $\forall M$ $\big[$$\big($ $M$ is transitive $\land$ $M\models\zfc_{n}^{-}$
				      $\land$ $\forall a\in M$ ( $R(a)$ $\leftrightarrow$ $M\models\left|a\right|\geq\aleph_{1}$
						      ) $\land$ $\A\in M$ $\big)$$\to$ $M\models$ ``$\A\models_{\LL} \fii$''$\big]$
		      .
	\end{enumerate}
	Note that since we consider transitive $M$'s, they are correct about the notion of well-order.

	\noindent
	2) The model class $Q_{R}=\left\{ \A\mid\A\cong(M,\in,a),M\text{ transitive and }|a| \geq \aleph_{1} \right\} $,
	in the language $\left(E,c\right)$ where $E$ is binary and $c$ is a constant, is $\Delta(\LL)$-definable:
	\begin{itemize}
		\item Consider the class of (2-sorted) models $\A=\left(A,B,E^{\A},c^{\A},<^{\A},F^{\A}\right)$
		      such that $E^{\A}$ is a binary relation on $A$ satisfying the Extensionality Axiom, $c^{\A}\in A$,
		      $<^{\A}$ is a well-order on $B$ (definable using $\W $),
		      $F^{\A}:\left(A,E^{\A}\right)\to(B,<^{\A})$ is an order-preserving function (a rank function) and $R(c^{\A})$ holds.
		      It is elementary
		      in $\LL(Q_{1},\W )$, and its projection yields the class of
		      well-founded models $\left(A,E^{\A},c^{\A}\right)$ such that $R(c^{\A})$.
		      The transitive collapse of such a model is of the form $(M,\in,\bar{c})$
		      where $\bar{c}$ is uncountable. So this is exactly $Q_{R}$.
		\item Consider the class of (2-sorted) models $\A=\left(A,B,E^{\A},c^{\A},<^{\A},G^{\A}\right)$
		      such that $E^{\A}$ is a binary relation on $A$, $c^{\A}\in A$,
		      $<^{\A}$ is a binary relation on $B$, $G^{\A}:B\to A$, and either $\neg R(c^{\A})$
		      holds or $E^{\A}$ fails to satisfy the Extensionality Axiom or [$<^{\A}$ is a linear order (on $B$) without least element
				      and $G^{\A}$ is order preserving].
		      It is elementary in $\LL(Q_{1},\W )$ (even $\LL(Q_{1})$\,),
		      and its projection yields exactly the complement model class of ${Q}_{R}$.
		      \qedhere
	\end{itemize}
\end{proof}
To conclude, symbiosis is not enough to imply that the $\Delta$ operation preserves the $\LL$-constructible universe:

\begin{theorem}\label{10}
	There is a logic $\LL$ which is symbiotic with a set theoretic predicate,
	such that  $Con(\zf)$ implies $Con(C(\LL)\ne C(\Delta(\LL)))$.
\end{theorem}

\begin{proof}
	Let $\LL^* = \LL(Q_1, \W )$.
	Note that $\Delta(\LL(Q_1)) \subset \Delta(\LL^*)$.
	So, as in the proof of Theorem \ref{thm:DeltaQ1}, we can force over $L$ to obtain a Cohen real $a$ and a forcing extension of $L$, say $W'$, such that $a \in C(\Delta(\LL^*))^{W'}$, while still $C(\LL^*)^{W'} = L$.
\end{proof}

\section{\texorpdfstring{$0^{\sharp}$ and the $\Delta$-operation}{0-sharp and the Δ-operation}} \label{sec:0sharp}

\sloppy Large cardinals give an even stronger result than the above Theorem~\ref{10}.
Let us denote $\Delta(\LL(Q_{1},\W ))$ by $\dqw$ (and similarly for $\Sigma,\Pi$) and $C(\Delta(\LL(Q_{1},\W )))$ by $\cqw$. We are indebted to Menachem Magidor for suggesting the following observation:
\begin{theorem}\label{thm:0sharpQ1W}
	If $0^{\sharp}$ exists, then  $0^{\sharp}\in \cqw$.
\end{theorem}

$0^{\sharp}$ can be defined in many equivalent ways,
and we present the definition most suitable for our purposes.
First recall the notion of indiscernibles.

\begin{definition}\label{def:indiscernibles}
	Let $\mathcal{M}=(M,E)$ be a model. A set $X \subseteq M$ is a \emph{set of indiscernibles for the model $(M,E)$} if for every $n \in \omega$, and every formula $\varphi\left(v_1, \ldots, v_n\right)$,
	$\mathcal{M} \vDash \varphi\left(\alpha_1, \ldots, \alpha_n\right) $ if and only if $\mathcal{M} \vDash \varphi\left(\beta_1, \ldots, \beta_n\right)$
	whenever $\alpha_1 E\alpha_{2}E \ldots E \alpha_n$ and $\beta_1 E\beta_{2} E \ldots E \beta_n$ are two $E$ increasing sequences of elements of $X$.
	We say that $X$ is a set of  \emph{ordinal indiscernibles for $(M,E)$} if, in addition,
	$(M,E) \vDash \text{``}x \text{ is an ordinal''}$ for every $x\in X$.
\end{definition}

\noindent Now, $0^{\sharp}$ is defined by the following fact.

\begin{fact}[Silver \cite{Silver-Thesis1971}]\label{fact:0sharp}
	If for \textbf{some} uncountable cardinal $\kappa$ there is a set of size $\kappa$ of ordinal indiscernibles for  $(L_{\kappa}, \in)$, then for \textbf{every} uncountable cardinal $\lambda$ there is a set of size $\lambda$ of ordinal indiscernibles for $(L_{\lambda},\in)$.

	Furthermore, in such a case, there is a set of first-order formulas, denoted $0^{\sharp}$, such that if $X$ is a set of ordinal indiscernibles for  $(L_{\kappa},\in)$ of size $\kappa$, and $\{\alpha_{k} \mid k<\omega\} \subseteq X$ is increasing, then
	\[ 0^{\sharp} = \{\varphi \mid (L_{\kappa},\in,\alpha_k)_{k<\omega} \models\varphi\}.\]
\end{fact}
The crux of the second part of this fact is that if the assumption ``for some uncountable cardinal $\kappa$ there is a set of size $\kappa$ of ordinal indiscernibles for  $(L_{\kappa}, \in)$'' holds, then the set $0^{\sharp}$ does not depend on $\kappa$ or the specific set of indiscernibles $X$.
Thus the assumption is commonly abbreviated as ``\emph{$0^{\sharp}$ exists}''.
In fact, $(L_{\kappa},\in)$ can be replaced with any well-founded model $(M,E)$ satisfying a large enough fragment of $\zfc+V=L$.
Let $<_{L}$ denote the uniformly definable well-order of $L$.
If $\mathcal{M}$ is a model of a large enough fragment of $\zfc$, satisfying the sentence $V=L$, then $<_{L}^{\mathcal{M}}$ is the interpretation of the definition of $<_{L}$ in $\mathcal{M}$.
Taking $\kappa=\omega_{1}$, note that $L_{\omega_{1}}$ is the unique (up to isomorphism) well-founded model of $V=L$ where  $<_{L}$ has order type $\omega_{1}$.
So, if $0^{\sharp}$ exists then for every well-founded model $\mathcal{M}=(M,E)$ of $V=L$
such that $<_{L}^{\mathcal{M}}$ has order type $\omega_{1}$,
and first order formula $\varphi(x_0,\ldots,x_{n-1})$ of set theory,
the following are equivalent:
\begin{enumerate}[label={(\arabic*)}]
	\item \label{enu:0sharp}  $ \varphi \in 0^{\sharp}$;
	\item \label{enu:sigma}  There is an uncountable $X \subseteq M$ such that
	      $X$ is a set of indiscernibles for  $\mathcal{M}$,
	      and there are $a_{0},\dots,a_{n-1} \in X$
	      such that $\mathcal{M} \models \varphi(a_{0},\dots,a_{n-1})$ ;
	\item \label{enu:pi}  For every uncountable $X \subseteq M$ such that
	      $X$ is a set of indiscernibles for  $\mathcal{M}$ and
	      for every $a_{0},\dots,a_{n-1} \in X$,
	      $\mathcal{M} \models \varphi(a_{0},\dots,a_{n-1})$.
\end{enumerate}

We will show that \ref{enu:sigma} and \ref{enu:pi} can be expressed as $\sqw$ and  $\pqw$, respectively, so that \ref{enu:0sharp} is in fact $\dqw$.

\begin{lemma}\label{lem:0sh-delta}
	If $0^{\sharp}$ exists then there is a formula $\Xi(u,v) \in \dqw$ such that for any admissible set $N\supseteq \omega_1^V$ and any $\mathcal{M}\in N$,
	\[
		\varphi \in 0^{\sharp} \iff (N,\in)\models_{\dqw} \Xi(\mathcal{M},\varphi). \qedhere
	\]
\end{lemma}

\begin{proof}
	We fix some notation, following Devlin's \cite[section A.I.9]{devlinConstructibility2017},
	where a distinction is made between formulas of the ``actual'' language of set theory with signature $\{\in\}$, denoted by \LST, and constructs which mimic these formulas \emph{inside} set theory, i.e. as specific sets.
	The result is a set denoted $\lst$, which we also treat as formulas, in a distinct signature $\{\epsilon\}$.
	Given a  set  $u$,  $\lst_{u}$ denotes formulas which include constant symbols for elements of $u$, such that the constant corresponding to $x\in u$ is denoted  $\mathring{x}$.
	We follow Devlin's convention where formulas of \LST\ are denoted by capital Greek letters, and their counterparts in $\lst$ by the corresponding lowercase letters.
	So for example if $\Phi(v)$ is an \LST\ formula with free variable $v$, and  $x\in u$, then there is a set  $\varphi(\mathring{x}) \in \lst_{u}$ corresponding to $\Phi(x)$.

	The results important to us are as follows:
	\begin{fact}\label{fact:lst}
		There are formulas of \LST\ as follows:
		\begin{enumerate}
			\item $\Fml(w,u)$ : $w$ is a formula of $\lst_{u}$;
			\item $\Fr(z,w)$ : $z$ is the set of free variables in the formula $w$;
			\item $\Sat(u, w)$ : $w$ is a sentence of $\lst_u$ which is true in the structure $( u, \in )$ under the canonical interpretation (i.e. with $\in$ interpreting $\epsilon$ and $x$ interpreting $\mathring{x}$ for each $x\in u$ ).
		\end{enumerate}
	\end{fact}

	These formulas can be modified to apply to general extensional structures $(u,e)$
	(where $e \subseteq u \times  u$)
	by taking care to replace every mention of $x \in y$ where  $x,y\in u$ by  $(x,y) \in e$.
	For example, a clause containing $(\exists x, y \in u)[(x \in y)...]$
	will be replaced by $(\exists x, y \in u)[(x,y) \in e ...]$.
	So by following the same constructions, we obtain formulas $\Fml^*(u,e,w)$, $\Fr^*(u,e,z,w)$, $\Sat^*(u,e,w)$ such that:

	\begin{fact}\label{fact:sat*}
		Let $\Phi\left(v_0, \ldots, v_n\right)$ be a formula of \LST, and $\varphi\left(v_0, \ldots, v_n\right)$ its counterpart in $\mathscr{L}$. Then for every well-founded extensional model $\mathcal{M}=(M,E)$ and $x_{0} ,\dots, x_{n} \in M$,
		\[
			\mathcal{M} \models
			\Phi(x_{0},\dots,x_{n}) \iff 	\Sat^*(M,E,\varphi(\mathring{x}_{0},\dots,\mathring{x}_{n}))
		\]
	\end{fact}

	Using this machinery, we will define a formula of \LST\ capturing the notion ``$X$ is a set of indiscernibles for $\mathcal{M}=(M,E)$'' (assuming $\mathcal{M}$ is a well-founded extensional model).
	We first describe some auxiliary notation.
	These notions should formally refer to a pair $(M,E)$ via variables $u,e$ as before, but for the sake of readability we will not always be explicit about it.
	So for example, when we talk about ordinals, we mean ``ordinals in the sense of $\mathcal{M}=(M,E)$''.
	We will also write e.g. $\psi \in \Fml^*(M,E)$ instead of $\Fml^*(M,E,\psi)$.

	\begin{itemize}
		\item An assignment (in $\mathcal{M}$) is a function whose  domain is a finite subset of variables (in the sense of $\mathcal{M}$).
		      We use $\sigma \in \Asn^*(X,Y)$ to denote that $\sigma$ is assignments with domain \emph{exactly} $X$ and range \emph{contained} in $Y$.
		      \footnote{This is in fact the same as $\sigma:X \to Y$, which we avoid to prevent confusion with the logical implication symbol $\to$ in the formulas presented below.}
		\item Given $\psi\in\Fml^*(M,E)$ and an assignment $\sigma$ such that the domain $\dom(\sigma)$ of $\sigma$ is $\Fr^*(M,E,\psi)$, we denote by $\psi[\sigma]$ the formula resulting from $\psi$ by substituting the free variables by the constants corresponding to the values of $\sigma$.
		      Note that this is an effective transformation on formulas in $\lst_{u}$.
		\item $v\in \ord^*(M,E)$ is the formula expressing that $v$ is an ordinal in the sense of  $(M,E)$.
		\item If $\sigma_{1},\sigma_{2}$ are assignments with ordinal values (i.e. functions whose domain is a finite subset of variables to the ordinals of $\mathcal{M}$), we denote by $\sigma_{1} \sim \sigma_{2}$ the statement:
		      \[
			      \dom(\sigma_{1})=\dom(\sigma_{2}) \land \forall v,u \in \dom(\sigma_{1})  [\sigma_{1}(v)<\sigma_{1}(u) \leftrightarrow \sigma_{2}(v) < \sigma_{2}(u)].
		      \]
	\end{itemize}

	\noindent	Now our desired formula, with predicate $X$ and parameter $\mathcal{M}$, denoted   $\Ind(X,\mathcal{M})$, is:
	\begin{gather*}
		\forall v[X(v)\to v\in \ord^*(M,E)] \land \forall \psi\in \Fml^*(M,E) \forall \sigma_{1},\sigma_{2}\in \Asn(\Fr^*(M,E,\psi),X)%
		\\
		[\sigma_{1} \sim \sigma_{2} \to (\Sat^*(M,E, \psi[\sigma_{1}]) \leftrightarrow \Sat^*(M,E, \psi [\sigma_{2}]))].
	\end{gather*}

	Let $\Gamma$ be the \LST\ sentence ``$V=L+\zfc^{-}_{n}$'' (for a large enough $n$),
	with $\gamma$ being the corresponding $\lst$-sentence.
	Let $\Upsilon(u,e,x,y)$ be the \LST\ formula such that whenever $(M,E)$ is a model of ``$V=L$'', $\Upsilon(M,E,x,y)$ defines its canonical well-order.
	Given a formula $\Psi=\Psi(x,y)$ which defines a linear order, let
	\[
		\Omega xy \Psi(x,y) = Q_{1}y \exists x \Psi(x,y) \land \forall y \neg Q_{1}x \Psi(x,y)
	\]
	be the $\LL(Q_{1})$ formula stating that the order defined by  $\Psi$ is $\omega_1$-like.

	Using this we can translate  (2) and (3) above (together with the needed assumptions on $\mathcal{M}$) to the following formulas, using $ \mathcal{M},\varphi$ as  first-order parameters:
	\begin{align*}
		2(\mathcal{M},\varphi) \quad
		 & \mathcal{M}=(M,E) \land  \W uv E(u,v) \land \Sat^*(M,E, \gamma)     \land \Omega xy \Upsilon(M,E,x,y) \\
		 & \to\ \exists X\subseteq M \big\{ Q_{1}v X(v) \land \Ind(X,\mathcal{M})                                \\
		 & \quad\land \exists \sigma\in \Asn(\Fr^*(M,E,\varphi), X)  [ \Sat^*(M,E, \varphi  [\sigma]) ]  \big\}
	\end{align*}
	\begin{align*}
		3(\mathcal{M},\varphi) \quad
		 & \mathcal{M}=(M,E) \land  \W uv E(u,v) \land \Sat^*(M,E, \gamma)     \land \Omega xy \Upsilon(M,E,x,y) \\
		 & \to \forall X\subseteq M \big\{Q_{1}v X(v) \land \Ind(X,\mathcal{M})                                  \\
		 & \quad\to \forall \sigma\in \Asn(\Fr^*(M,E,\varphi), X)  [ \Sat^*(M,E, \varphi[\sigma]) ] \big\}
	\end{align*}
	So, replacing $\mathcal{M},\varphi$ with variables $u,v$ (and making all the implicit definitions explicit), the above formulas are equivalent for every  $u,v$. The first formula is $\Sigma(Q_{1},\W)$ and the second $\Pi(Q_{1},\W)$, hence they are both $\Delta(Q_{1},\W )$.
	Hence, if we denote by $\Xi(u,v)$ the $\dqw$ formula represented by this pair, we get that indeed for any admissible $N\supseteq \omega_1^V$ and any $\mathcal{M}\in N$,
	\[
		\varphi \in 0^{\sharp} \iff (N,\in)\models_{\dqw} \Xi(\mathcal{M},\varphi). \qedhere
	\]
\end{proof}

\begin{proof}[Proof of theorem \ref{thm:0sharpQ1W}]
	Assume $0^{\sharp}$ exists and let $\cqw = \bigcup L'_{\alpha}$ be the levels of the construction, as in definition \ref{def:C(L)}.
	Let $\alpha$ be such that $L_{\omega_{1}} \in L'_{\alpha}$ (here $\omega_{1}=\omega_{1}^{V}$).
	For every first order formula $\varphi(v_{0},\dots,v_{n})$ in the signature $\{\epsilon\}$, and $x_{0},\dots,x_{n} \in L_{\omega_{1}}$ we have
	\begin{gather*}
		\Sat^*(L_{\omega_{1}}, \in \restriction_{L_{\omega_1}},\varphi(\mathring{x}_{0},\dots,\mathring{x}_{n})) \iff \\
		(L_{\omega_{1}}, \in \restriction_{L_{\omega_1}}) \models \Phi(x_{0},\dots,x_{n}) \iff \\
		(L'_{\alpha},\in) \models \Sat^*(L_{\omega_{1}}, \in \restriction_{L_{\omega_1}},\varphi(\mathring{x}_{0},\dots,\mathring{x}_{n})).
	\end{gather*}

	\noindent For short, we denote the model $(L_{\omega_{1}}, \in \restriction_{L_{\omega_1}})$ simply by $L_{\omega_{1}}$.
	So for $X \subseteq L'_{\alpha}$, we will get that
	\[
		\Ind(L_{\omega_{1}},X) \iff (L'_{\alpha},\in,X) \models \Ind(L_{\omega_{1}},X).
	\]
	Note that here we are considering first-order satisfaction.
	So now, if we consider satisfaction in $\lqw$ and  $\dqw$, we get that for every $\varphi$
	\begin{align*}
		2(L_{\omega_{1}},\varphi) & \iff (L'_{\alpha},\in) \models_{\lqw} 2(L_{\omega_{1}},\varphi)  \\
		3(L_{\omega_{1}},\varphi) & \iff (L'_{\alpha},\in) \models_{\lqw} 3(L_{\omega_{1}},\varphi),
	\end{align*}
	and by the result of the previous lemma, we get
	\[
		(L'_{\alpha},\in) \models_{\dqw} \Xi(L_{\omega_{1}},\varphi)	\iff \varphi \in 0^{\sharp}.
	\]
	Hence
	\[
		0^{\sharp}=	\{  \varphi    \mid (L'_{\alpha},\in) \models_{\dqw} \Xi(L_{\omega_1},\varphi)\} \in \cqw. \qedhere
	\]
\end{proof}

\section{Conclusion and open questions}

Let us summarize what has been obtained:
\begin{theorem*}\label{thm:summary}\phantom{}

	\begin{enumerate}
		\item $\zf \vdash C(\LL(Q_{0})) = C(\Delta(\LL(Q_{0}))) = L$ (while $\LL(Q_{0}) \ne \Delta(\LL(Q_{0}))$).
		\item $\zfc \vdash C(\LL^{2}) = C(\Delta(\LL^{2})) = \hod$ (while $\LL^{2} \ne \Delta(\LL^{2})$).
		\item It is consistent with $\zfc$ that $C(\LL(H))=\hod_{1} \ne \hod = C(\Delta(\LL(H)))$.
		\item It is consistent with $\zfc$ that $C(\LL(Q_{1}))=L \ne C(\Delta(\LL(Q_{1})))$ .
		\item It is consistent with $\zfc$ that $C(\LL(Q_{1},\W ))=L \ne C(\Delta(\LL(Q_{1},\W )))$, where $\LL(Q_{1},\W )$ is symbiotic with the predicate $R_{\aleph_{1}}$.
		\item If  $0^{\sharp}$ exists, then  $C(\LL(Q_{1},\W ))=L \ne C(\Delta(\LL(Q_{1},\W ))) \ni 0^{\sharp}$.
	\end{enumerate}

\end{theorem*}

This leaves the following question open:

\begin{question}\label{que:pres-cond} 
	Are there necessary and/or sufficient conditions on the logic $\LL$ (besides the trivial $\LL=\Delta(\LL)$) which imply  $\zf \vdash C(\LL)=C(\Delta(\LL))$?
\end{question}

Dual to this would be trying to identify properties of a logic that would \emph{allow} $C(\LL) \ne C(\Delta(\LL))$:
\begin{question}\label{que:diff-cond} 
	Are there necessary and/or sufficient conditions on the logic $\LL$ which imply the consistency of $C(\LL) \ne C(\Delta(\LL))$?
\end{question}

The path to answering both these questions would benefit from analysing additional examples of logics and their $\Delta$-extensions.
Such examples may include the logics obtained from Magidor-Malitz or cofinality quantifiers ($\LL(Q^{\mathrm{MM}})$ and $\LL(Q^{\mathrm{cf}}_{\alpha})$), stationary logic $\LL(\mathtt{aa})$, and so on.

A further question regarding the theory of models constructed from logics obtained by the $\Delta$-operation is the status of AC in them.
As we mentioned in the beginning (page \pageref{AC}), the known methods of proving that AC holds in models of the form $C(\LL)$ requires  $\LL$ to have recursive syntax, which $\Delta(\LL)$ is not known to have in general.
In particular, we do not know whether the models $\cq$ and $\cqw$ provably satisfy AC, and even whether AC holds in the specific constructions we used here.
Answering these questions will require a deeper analysis of the ways forcing affects truth in the logics $\dq$ and $\dqw$, which is beyond the scope of this paper.
Hence the following is open:

\begin{question}\label{que:AC}
	What is the status of the Axiom of Choice in models of the form $C(\Delta(\LL))$?
\end{question}

Another direction for further research is the impact of large cardinals on $C(\Delta(\LL))$.
Theorem \ref{thm:0sharpQ1W} gives an initial result showing that large cardinals do affect the expressive power of logics obtained by the $\Delta$-operation.
Hence we may ask a somewhat broad question:

\begin{question}\label{que:largecards}
	What is the effect of large cardinals on the relationship between $C(\LL)$  and $C(\Delta(\LL))$?
\end{question}

On a more concrete level, the models $\cq$ and $\cqw$ require further investigation.
For example, a first step would be to answer the following:
\begin{question}\label{que:specific-models}
	\begin{enumerate} \thmenumhspace
		\item Is it consistent that $\cq \ne \cqw$?
		\item Do the models $\cq, \cqw$ satisfy $AC$?
		\item Do the models $\cq, \cqw$ provably contain all the reals?
	\end{enumerate}
\end{question}
\noindent
Another set of questions regards the question of idempotence: for a logic $\LL$, we say that  $C(\LL)$ is \emph{provably idempotent} if it is provable that $C(\LL)^{C(\LL)}=C(\LL)$;
and \emph{consistently idempotent (and $\ne L$)} if it is consistent that $C(\LL)^{C(\LL)}=C(\LL)\ne L$. We will usually omit the phrase in parentheses.
Note that the last statement is equivalent to ``it is consistent that $V=C(\LL)\ne L$ holds''.
Of course it is only interesting in the case that consistently, $C(\LL) \ne L$, as $L$ is provably idempotent. Other examples of provable idempotence are models constructed from infinitary logics $C(\LL_{\kappa,\kappa})$ \cite{changLkk}, but this is the exception rather than the rule, as  models such as $\hod, C^*, C(\aaa)$ are only consistently idempotent (see \cite{McAloon1,mcaloon2,KMV1,KMV2,Yaar-IteratingCofinalityoConstructible2023,Yaar-IteratedClubShooting2025}).

It is likely that the constructions we produced here for the logics $\cq$ and  $\cqw$ result in models where idempotence fails, as, intuitively, we only code the truth of some fact (for example which trees are not Aronszajn), but not the witnesses to this fact (the generic branches).
However proving that this is in fact the case requires further research, and we leave the following open:
\begin{question}\label{que:idempotence}
	Are the models $\cq$ and  $\cqw$ consistently or even provably idempotent?
\end{question}

\noindent
Many more questions may be asked, which we hope to address in further research.

\subsection{The $\Delta_{1}^{1}$ operation}\label{subsec:delta11}
Finally, we mention a notion closely related to the $\Delta$-operation -- the $\Delta_{1}^{1}$-operation -- defined in the same way, except that when expanding the vocabulary we only allow adding new predicates, not new sorts.
This results in a weaker closure operation on logics, as for example $\Delta_{1}^{1}(\LL^{2}) = \LL^{2} \ne \Delta(\LL^{2})$ (for more on this, see \cite{Vaananen-DextensionHanfnumbers1983}).
Note that in this case,
\[
	\hod = C(\LL^{2}) = C(\Delta_{1}^{1}(\LL^{2})) = C(\Delta(\LL^{2})).
\]
However the equality $C(\Delta_{1}^{1}(\LL))=C(\Delta(\LL))$ is not true in general, as can be seen in the following example.
Recall Theorem \ref{thm:Henkin-DHenkin} where we show that consistently $C(\LL(H)) \ne C(\LL(\Delta(H)))$.
\begin{proposition}\label{prop:delta11Henkin}
	$C(\Delta_{1}^{1}(\LL(H)))=C(\LL(H))$
\end{proposition}

\begin{proof}
	The direction $\supseteq$ is clear.
	Now note that a $\Sigma_{1}^{1}(L(H))$-formula can be expressed in the form $\exists R_{1}\dots \exists R_{n} \Phi$ with $\Phi\in L(H)$,
	and as we noted in the proof of Theorem \ref{thm:HenkinHOD1},
	in the presence of a pairing function we can express the existential quantifiers in the terms of the Henkin quantifier.
	So we can inductively prove the other direction.
\end{proof}

\begin{corollary}\label{cor:delta11Henkin}
	It is consistent with $\zfc$ that $C(\Delta_{1}^{1}(\LL(H))) \ne C(\Delta(\LL(H)))$.
\end{corollary}

Nevertheless, the results in sections \ref{sec:Q1}-\ref{sec:0sharp} apply to the $\Delta_{1}^{1}$-operation as well,
as the formulas used to separate the models only require adding predicates, not new sorts.
Thus the following holds:

\begin{theorem}\label{thm:Delta11}\phantom{}
	\begin{enumerate}
		\item It is consistent with $\zfc$ that $C(\LL(Q_{1}))=L \ne C(\Delta_{1}^{1}(\LL(Q_{1})))$.
		\item It is consistent with $\zfc$ that $C(\LL(Q_{1},\W ))=L \ne C(\Delta_{1}^{1}(\LL(Q_{1},\W )))$, where $\LL(Q_{1},\W )$ is symbiotic with the predicate $R_{\aleph_{1}}$.
		\item If  $0^{\sharp}$ exists, then  $C(\LL(Q_{1},\W ))=L \ne C(\Delta_{1}^{1}(\LL(Q_{1},\W ))) \ni 0^{\sharp}$.
	\end{enumerate}
\end{theorem}

\bibliographystyle{amsalpha}
\bibliography{Bibliography}

\end{document}